\documentclass[12pt]{amsart}

\usepackage{amsfonts,amssymb,amsmath,textcomp}

 2

\usepackage{amsmath}
\usepackage{amssymb}
\usepackage{amsthm}
\usepackage[mathscr]{euscript}
\usepackage{textcomp}
\usepackage{amsmath}
\usepackage{amssymb}
\usepackage{amsthm}
\newtheorem{thm}{Theorem}[section]

\newtheorem{lem}{Lemma}
\newtheorem{rem}{Remark}

\newcommand{\Z}{\mathbb{Z}}
\newcommand{\Q}{\mathbb{Q}}

\newcommand{\C}{\mathbb{C}}
\begin{document}
\title
{Distribution of prime ideals across ideal classes in the class groups}
\author{Prem Prakash Pandey	}

\address[Prem Prakash Pandey]{School of Mathematical Sciences,
                                            NISER Bhubaneswar (HBNI),
                                            Jatani, Khurda-650 052, India.}
                                            
\email{premshivaganga@gmail.com}


\date{\today}


\maketitle
\begin{abstract}
In this article we investigate the distribution of prime ideals of residue degree bigger than one across the ideal classes in the class group of a number field $L$. A criterion for the class group of $L$ being generated by the classes of prime ideals of residue degree $f>1$ is provided. Further, some consequences of this study on the solvability of norm equations for $L/\mathbb{Q}$ and on the problem of finding annihilators for relative extensions are discussed.  
\end{abstract}
\section{Introduction}
In this article, $L$ and $K$ denote number fields such that $K$ is a subfield of $L$. We assume that $L/K$ is Galois and write $G:=Gal(L/K)$. Let $\mathfrak{p}$ be a prime ideal of $L$, $\mathbf{p}$ denote the prime ideal of $K$ below $\mathfrak{p}$. The residue degree $[ \mathbb{O}_L/\mathfrak{p}:\mathbb{O}_K / \mathbf{p}]$ will be denoted by $res^L_K(\mathfrak{p})$, and when $K=\mathbb{Q}$ we should simply write $res(\mathfrak{p})$.  Also we let $\ell$ be an odd prime number and fix a primitive $\ell^{th}$ root of unity $\zeta_{\ell}$. Let $\Q(\zeta_{\ell})$ denote the subfield of $\C$ obtained by adjoining $\zeta_{\ell}$ to $\Q$. An important result in algebraic number theory is the following theorem, which is one of many density theorems. 
\begin{thm}\label{TA}[ Theorem 4.6, \cite{GJJ}]
Every ideal class in the class group of $L$ contains infinitely many prime ideals $\mathfrak{p}$ of residue degree one, that is, $res(\mathfrak{p})=1$.
\end{thm}
Let $\mathfrak{c}$ be an ideal class in the class group of $L$ and $\mathfrak{p}$ be a prime ideal in $\mathfrak{c}$. If $p$ denotes the rational prime lying below $\mathfrak{p}$ then, by Theorem \ref{TA}, one may assume that $p$ is unramified in $L/\mathbb{Q}$ and the following factorization holds
$$p \mathbb{O}_L= \prod_{\sigma \in Gal(L/\mathbb{Q})} \sigma(\mathfrak{p}).$$ 
Thus, for $N= \sum_{\sigma \in Gal(L/\mathbb{Q})} \sigma$ we have
$$\mathfrak{c}^N=[\mathfrak{p}]^N=[(p)],$$ which is trivial. This shows that $N$ annihilates the class group $C\ell(L)$. But $N$ is not very useful as annihilator, as in applications of annihilators, mostly one uses $(1-\sigma)\theta$ for annihilation, with $\theta$ being an annihilator and $\sigma $ being the complex conjugation (e.g. see the works of Mihailescu proving Catalan's conjecture in \cite{RS} or \cite{YB}). 

There are various accounts of finding elements in the group ring $\Z[G]$ which annihilate the class group $C\ell(L)$ of $L$. When $K=\Q$, the Stickelberger theorem is a celebrated result in this direction (see \cite{SL} or \cite{LW}). In full generality (that is when $K$ is arbitrary) there is no description of elements in $\Z[G]$ which annihilate the class group $C\ell(L)$ (though there are some results in special cases, see \cite{SD} or \cite{JWS, JT}). On the other hand, if $K$ has class number one, then it is easy to see (for example, as illustrated after Theorem \ref{TA}), that the $G$-trace $N=\sum_{\sigma \in G} \sigma$ annihilates the class group $C\ell(L)$. The perspective taken in this article is to explore an analogue of Theorem \ref{TA} for higher residue degrees and possible consequences. For extension $L/K$, we define the set 
\begin{eqnarray*}
\mathscr{R}_K^L & := & \{f \in \mathbb{N}: C\ell(L) \mbox{ is generated by classes of unramified prime ideals }\\ & & \mathfrak{p} \mbox{ with } res^L_K(\mathfrak{p})=f\}.
\end{eqnarray*}

From Theorem \ref{TA} it follows that $1 \in \mathscr{R}_K^L$ for any extension $L/K$. For $L=\Q(\zeta_{\ell})$ and $K=\Q$, Kummer had proved this using only algebraic tools (see \cite{Kum} or chapter 9 in \cite{RS}). This algebraic proof has been further extended by Lenstra and Stevenhagen in more general set-up (see \cite{LS}). To the best of our knowledge, the following question (which can be seen as a generalization of Theorem \ref{TA}) has not been addressed in the literature.

\noindent{\bf Question 1:} When does the set $\mathscr{R}_K^L$ has more than one element?

The question is interesting only when the class number of $L$ is bigger than $1$. When $G$ is cyclic and $K$ has class number $1$, then each element of $\mathscr{R}_K^L$ gives rise to an annihilator of the class group $C\ell(L)$. Let $f \in \mathscr{R}_K^L$ and let $G'$ be the unique subgroup of $G$ of order $f$. If $\{\sigma_1, \ldots, \sigma_g\}$ is a complete set of representatives of the elements of $G/G'$, then we put $\theta_f =\sum_{i=1}^g \sigma_i$ and prove the following theorem.

\begin{thm}\label{T3}
Consider a cyclic extension $L/K$ of number fields. Then at least one of the following holds:\\
(1) the class number of $K$ is bigger than one,\\
(2) the element $\theta_f$ annihilates $C\ell(L)$ for each $f\in \mathscr{R}_K^L$.
\end{thm}
For $f=1$ we have $\theta_1=N$, the $G$-trace. Using Theorem \ref{T3}, in section 3, we shall show that for the fields $L=\Q(\zeta_{23})$ and $K=\mathbb{Q}$ we have $\mathscr{R}_K^L=\{1\}$. Next, we give a criterion to determine if a positive integer $f$ is in $\mathscr{R}_K^L$ or not. For this, let $H:=H(L)$ be the Hilbert class field of $L$ (maximal unramified abelian extension of $L$). Then $H/K$ is Galois (see Lemma \ref{L2}) and we have a short exact sequence
\begin{equation}\label{D}
0 \longrightarrow Gal(H/L) \longrightarrow Gal(H/K) \longrightarrow Gal(L/K) \longrightarrow 0.
\end{equation}
For any $\sigma \in Gal(L/K)$ we use $\tilde{\sigma}$ for any lift of $\sigma$ to $H$, that is, $\tilde{\sigma} \in Gal(H/K)$ and $\tilde{\sigma}\shortmid_L=\sigma$. Now we are in a position to state the criterion.
\begin{thm}\label{T1}
For an integer $f>1$, the ideal class group $C\ell(L)$ is generated by the classes of prime ideals $\mathfrak{p}$ with $res^L_K(\mathfrak{p})=f$ if and only if the set $$\{(\tilde{\sigma})^f: \sigma \in Gal(L/K) \mbox{ and the order of }\sigma \mbox{ is }f\}$$ generates the group $Gal(H/L)$.
\end{thm}
We give a proof of Theorem \ref{T1} in the next section. The proof also demonstrates that the size of the subgroup generated by the set $\{(\tilde{\sigma})^f: \sigma \in Gal(L/K) \mbox{ is of order } f \}$ measures the size of the subgroup of the class group which is generated by the classes of prime ideals $\mathfrak{p}$ with $res^L_K(\mathfrak{p})=f$. This can be exploited to study the (absolute) norm equations. The study of solvability of norm equations for number fields and algorithms to determine solutions to norm equations are well pursued (see \cite{DG,VA, FJP, BN} and references in there). The special case $\mathscr{R}_K^L=\{1\}$ has an immediate bearing on the solvability of the norm equations. In this direction we give the following sufficient condition for the solvability of norm equations. 
\begin{thm}\label{T2}
Let $L$ be a number field whose class number is a prime number. If $\mathscr{R}_{\mathbb{Q}}^L=\{1\}$ then the norm equation
\begin{equation}\label{e0}
|N_{L/\mathbb{Q}}(x)|=a, a\geq 0
\end{equation}
is solvable whenever the prime divisors $p$ of $a$ are unramified, of residue degree bigger than $1$ and $\upsilon_p(a)$ is a multiple of the residue degree of $p$. Here $\upsilon_p(a)$ is the $p$-adic valuation of $a$, that is, the highest power of $p$ which divides $a$.
\end{thm}

The only reason for considering the absolute norm equation in Theorem \ref{T2} is that we can not say, in general, whether $-1$ is a norm or not. Thus, if the extension $L/\mathbb{Q}$ is not totally real then in Theorem \ref{T2} we can replace equation (\ref{e0}) by 
$$N_{L/\mathbb{Q}}(x)=a.$$
In section 3, Theorem \ref{T2} is used to give a very concrete description of the solvability of the norm equations for the extension $\mathbb{Q}(\zeta_{23})/ \mathbb{Q}$ in a very elementary way (see Theorem \ref{T4}).

\section{Proofs}
We begin this section with a proof of Theorem \ref{T3}.
\begin{proof}
[\bf Proof of Theorem \ref{T3}]
If the class number of $K$ is bigger than one then nothing to prove. So we assume that the class number of $K$ is one. Let $\mathfrak{c}$ be an ideal class in $C\ell(L)$. Then 
$$\mathfrak{c} =[\mathfrak{p}_1] [\mathfrak{p}_2] \ldots [\mathfrak{p}_t]$$ for unramified prime ideals $\mathfrak{p}_1,  \ldots ,\mathfrak{p}_t  $ of residue degree $f$. Thus it is enough to prove that $\theta_f$ annihilates all the unramified prime ideals of residue degree $f$.\\ 
Let $\mathfrak{p}$ be an unramified prime ideal of the residue degree $f$ in $L$ and let $\mathbf{p}$ be the prime ideal of $K$ lying below $\mathfrak{p}$. Let $D_{\mathfrak{p} }$ denote the decomposition group at $\mathfrak{p}$. Then $D_{\mathfrak{p} }$ is the unique subgroup of $G$ of order $f$ and it does not depend on $\mathfrak{p}$. If $\{ \sigma_1, \ldots, \sigma_g \}$ is a complete set of representatives of $G/D_{\mathfrak{p}}$, then $\{\sigma_i(\mathfrak{p}): 1 \leq i \leq g \}$ is the set of all conjugates of $\mathfrak{p}$. Thus the factorization of $\mathbf{p} \mathbb{O}_L$ is given by 
$$ \mathbf{p} \mathbb{O}_L= \prod_{i=1}^g \sigma_i(\mathfrak{p}).$$ 
Since $\theta_f$ is a multiple of $\sum_{i=1}^g \sigma_i$ by some $\tau \in G$, it follows that
$$\mathfrak{p}^{\theta_f}=\tau(\mathbf{p} \mathbb{O}_L)$$
is principal. Thus $\theta_f$ annihilates the class group $C\ell(L)$.
\end{proof}
To prove Theorem \ref{T1} we need some preliminaries. We begin with the following elementary lemma (as was indicated in section 1).
\begin{lem}\label{L2}
If $L/K$ is a Galois extension of number fields and $H$ is the Hilbert class field of $L$, then $H/K$ is Galois.
\end{lem}
\begin{proof}
We fix an algebraic closure $\bar{\Q}$ of $\Q$. Let $\sigma \in Gal(\bar{\Q}/ K)$, then $\sigma (L)=L$, as $L/K$ is Galois. Since $H/L$ is maximal unramified hence so is $\sigma (H) / \sigma (L)$. Consequently $\sigma (H) \subset H$, proving that $H/K$ is Galois.
\end{proof}

For an unramified prime ideal $\mathfrak{p}$ in $L$, let 
%we use $res_{\mathfrak{p}}^{L/K}$ to denote the degree $[\mathbb{O}_{L}/ \mathfrak{p} : \mathbb{O}_{K}/ {\mathfrak{p} \cap \mathbb{O}_{K}}]$. Further if $\mathfrak{p}$ is unramified then we use 
$\left( \frac{\mathfrak{p}}{L/K} \right)$ denote the Frobenius of $\mathfrak{p}$ with respect to the extension $L/K$. If $L/K$ is abelian then we also write $\left( \frac{\mathbf{p}}{L/K} \right)$ for $\left( \frac{\mathfrak{p}}{L/K} \right)$, where $\mathbf{p}=\mathfrak{p} \cap \mathbb{O}_{K}$. From the definition of the Frobenius, we have the following lemma (see page 127 in \cite{GJJ}).
\begin{lem}\label{L1}
Let $L/K$ be a Galois extension of number fields and let $F$ be an intermediate field such that $F/K$ is Galois. Then for any unramified prime ideal $\mathfrak{p} \mbox{ of }L$ one has $\left( \frac{\mathfrak{p}}{L/K} \right)_{|F}=\left( \frac{\mathfrak{p} \cap \mathbb{O}_F}{F/K} \right)$.
\end{lem}
Next we recall the Chebotar\"{e}v density Theorem (see \cite{SL1}). For any $\sigma \in Gal(L/K)$, let $P_{L/K}(\sigma)$ denote the set of prime ideals $\mathbf{p}$ in $K$ such that there is a prime ideal $\mathfrak{p}$ of $L$ above $\mathbf{p}$ such that $\left( \frac{\mathfrak{p}}{L/K} \right) = \sigma $. 
\begin{thm} \label{DT}
[ Chebotar\"{e}v Density Theorem] Let $\sigma \in Gal(L/K)$ and $C_{\sigma}$ stand for the conjugacy class of $\sigma$ then the density of $P_{L/K}(\sigma)$ is $|C_{\sigma}|/[L:K]$.
\end{thm}

%Now we can furnish a proof of the Theorem \ref{T1}.
\begin{proof}[\bf Proof of Theorem \ref{T1}]
For any $\sigma \in Gal(L/K)$ of order $f$, it is immediate to see that $$(\tilde{\sigma})^f \shortmid_L=id,$$ and thus $(\tilde{\sigma})^f \in Gal(H/L)$.\\
Assume that the set $\{(\tilde{\sigma})^f: \sigma \in Gal(L/K) \mbox{ and the order of }\sigma \mbox{ is }f\}$ generates the group $Gal(H/L)$. Let $\mathfrak{c}$ be an ideal class in $C\ell(L)$ and $\tau$ be the corresponding element in $Gal(H/L)$ under the Artin isomorphism between $C\ell(L)$ and the Galois group $Gal(H/L)$. Then there is a prime ideal $\mathfrak{p}$ in $\mathfrak{c}$ such that $$\left( \frac{\mathfrak{p}}{H/L} \right)=\tau.$$
From our assumption, there are elements $\sigma_1, \ldots, \sigma_r$ in $Gal(L/K)$ of order $f$ such that $$\tau=(\tilde{\sigma_1})^f \ldots (\tilde{\sigma_r})^f.$$  
By Chebotarev density theorem, for each $i, 1 \leq i \leq  r$ there exists prime ideal $\wp_i$ of $H$ such that $$\left( \frac{\wp_i}{H/K} \right)=\tilde{\sigma_i}.$$ 
Let $\mathfrak{p}_i=\wp_i \cap \mathbb{O}_L$ for $i=1, \ldots, r$. Then  from Lemma \ref{L1} it follows that $\left( \frac{\mathfrak{p}_i}{L/K} \right)=\sigma_i$, for $i=1, \ldots, r $. Since the order of $\sigma_i$ is $f$, we conclude that the residue degree of $\mathfrak{p}_i$ is $f$. Next we note that
$$\left( \frac{\wp_i}{H/L} \right)=(\tilde{\sigma_i})^f.$$
Since $H/L$ is abelian, we get
$$\left( \frac{\mathfrak{p}_i}{H/L} \right)=(\tilde{\sigma_i})^f.$$
This leads to 
$$\left( \frac{\mathfrak{p}}{H/L} \right)=\left( \frac{\mathfrak{p}_1}{H/L} \right) \ldots \left( \frac{\mathfrak{p}_r}{H/L} \right).$$ From the above equality, it follows that $\mathfrak{c}=[\mathfrak{p}_1] \ldots [\mathfrak{p}_r]$, as desired.\\
Conversely assume that the ideal class group $C\ell(L)$ is generated by the classes of prime ideals of residue degree $f$.\\ 
Let $\tau \in Gal(H/L)$ and let $\mathfrak{c}$ be the ideal class corresponding to $\tau $ under the Artin isomorphism. By our assumption, there are prime ideals $\mathfrak{p}_1, \ldots, \mathfrak{p}_r$ of residue degree $f$ such that $\mathfrak{c}=[\mathfrak{p}_1] \ldots [\mathfrak{p}_r]$. Put $\sigma_i=\left( \frac{\mathfrak{p}_i}{L/K} \right)$, for $i=1, \ldots, r$. Then it follows immediately that $\sigma_i$ is of order $f$ and $\tau=(\tilde{\sigma_1})^f \ldots (\tilde{\sigma_r})^f$. This proves the Theorem.
\end{proof}

\begin{rem}\label{r1}
From the proof of the Theorem \ref{T1}, it is immediate that the size of the subgroup generated by the set $\{(\tilde{\sigma})^f: \sigma \in Gal(L/K) \mbox{ is of order } f \}$ measures the size of the subgroup of the class group $C\ell(L)$ which is generated by the classes of prime ideals of residue degree $f$. In particular, if the class number of $L$ is prime and $f \not \in \mathscr{R}_K^L$ then all the prime ideals of residue degree $f$ are principal in $L$.
\end{rem}

\begin{proof}
[\bf Proof of Theorem \ref{T2}]
Let $p$ be a prime of residue degree $f>1$ which is unramified in $L$. We shall show that there is an element $\alpha \in L$ such that $|N_{L/K}(\alpha)|=p^f$.\\
Since the class number of $L$ is a prime number and $\mathscr{R}_{\mathbb{Q}}^L=\{1\}$, the subgroup generated by the set $$\{(\tilde{\sigma})^f: \sigma \in Gal(L/K) \mbox{ and the order of }\sigma \mbox{ is }f\}$$ is trivial. Consequently, from the Remark \ref{r1}, all the prime ideals $\mathfrak{p}$ of $L$ above $p$ are principal. Let $\mathfrak{p}$ be a prime ideal of $L$ dividing $p$ and $\alpha$ be a generator of $\mathfrak{p}$. Then we have
$$|N_{L/\mathbb{Q}}(\alpha)|=p^f.$$
The theorem follows at once from the multiplicative property of the norm map.\\
\end{proof}

%\begin{rem}
%Along the lines of the proof of Theorem \ref{T2} we can prove that if the class number of $L$ is a prime number and $\mathscr{R}^L_{\mathbb{Q}}=\{1\}$ then the Hasse norm principle (see \cite{BN}) holds for the extension $L/\mathbb{Q}$ whenever there is a subfield $K$ of $L$ with class number one and cyclic Galois group $Gal(L/K)$.
%\end{rem}

\section{An example}
In this section, we consider the fields $L=\Q(\zeta_{23})$ and $K= \Q$ and show that $\mathscr{R}_K^L=\{1\}$. Note that, the condition `unramified' in the definition of $\mathscr{R}_K^L$ is redundant in this case. If $\mathfrak{p}$ is a prime ideal in $\Q(\zeta_{\ell})$ of residue degree $f$ then $f|22$, and thus $f \in \{1,2,11,22\}$. Since the class number of $\Q(\zeta_{23})$ is bigger than $1$, it follows that $f=22$ is not possible. So it remains to show that $f=2$ or $f=11$ is not possible.

Before proceeding further, we recall some results on cyclotomic fields which will be needed. Let $h_{\ell}^+$ and $h_{\ell}$ denote the class numbers of $\Q(\zeta_{\ell}+\zeta_{\ell}^{-1})$ and $\Q(\zeta_{\ell})$ respectively and put  $h_{\ell}^{-}=h_{\ell}/h_{\ell}^+$. Let $G$ be the Galois group $Gal(\Q(\zeta_{\ell})/ \Q)$ and $S$ denote the Stickelberger ideal in $\Z[G]$. It is well known that $[\Z[G]:S]=h_{\ell}^{-}$. We now describe a basis of $S$ (see chapter 9 in \cite{RS}).\\
For each $a \geq 1$ with $(a,\ell)=1$ we define 
$$\theta_a=\sum_{i=1}^{\ell-1} \lfloor\frac{ai}{\ell}\rfloor \sigma_i^{-1},$$ where $\sigma_i:\zeta_{\ell} \longmapsto \zeta_{\ell}^i$ and $\lfloor x \rfloor$ denotes the largest integer not bigger than $x$. Further we put $f_i=\theta_{i+1}-\theta_i$. Then we have following Theorem due to Kummer.
\begin{thm}\label{TA3}
[Theorem 9.3, \cite{RS}] The elements $f_1, \ldots, f_{(p-1)/2}$ together with the $G$-trace $N=\sum_{\sigma \in G} \sigma$ forms a $\Z-basis$ of $S$. 
\end{thm}
We recall following fact (see Theorem 1.1 in \cite{JM}).
\begin{thm}\label{TA4}
We have $h_{\ell}^+=1$ for $\ell<100$.
\end{thm}
Now we fix $\ell=23$, we have $h_{23}=h_{23}^{-}$. Thus, if $\theta \in \Z[G]$ annihilates the class group of $\Q(\zeta_{23})$ then $\theta $ must lie in $S$. In the case of $\Q(\zeta_{23})$, we have
$$f_1=\sum_{i \in I_1} \sigma_i, \mbox{ where } I_1=\{2,16,5,20,13,19,9,17,15,11,22  \},$$
$$f_2=\sum_{i \in I_2} \sigma_i, \mbox{ where } I_2=\{3,18,7,21,13,19,9,17,15,11,22  \},$$
$$f_3=\sum_{i \in I_3} \sigma_i, \mbox{ where } I_3=\{4,10,2,16,5,20,9,17,15,11,22  \},$$
$$f_4=\sum_{i \in I_4} \sigma_i, \mbox{ where } I_4=\{2,16,5,20,13,19,9,17,15,11,22  \},$$
$$f_5=\sum_{i \in I_5} \sigma_i, \mbox{ where } I_5=\{6,3,18,2,16,13,19,9,15,11,22  \},$$
$$f_6=\sum_{i \in I_6} \sigma_i, \mbox{ where } I_6=\{10,7,21,5,20,19,9,17,15,11,22  \},$$
$$f_7=\sum_{i \in I_7} \sigma_i, \mbox{ where } I_7=\{8,4,18,2,16,20,13,9,17,11,22  \},$$
$$f_8=\sum_{i \in I_8} \sigma_i, \mbox{ where } I_8=\{3,21,16,5,13,19,9,17,15,11,22  \},$$
$$f_9=\sum_{i \in I_9} \sigma_i, \mbox{ where } I_9=\{14,10,7,2,5,20,19,17,15,11,22  \},$$
$$f_{10}=\sum_{i \in I_{10}} \sigma_i, \mbox{ where } I_{10}=\{18,21,16,20,13,19,9,17,15,11,22  \},$$
$$f_{11}=\sum_{i \in I_{11}} \sigma_i, \mbox{ where } I_{11}=\{12,6,4,3,7,2,5,13,9,15,22  \}.$$
If the class group of $\Q(\zeta_{23})$ is generated by prime ideals of residue degree $11$ then by Theorem \ref{T3} the element $\theta_{11}$ is an annihilator of the class group of $\Q(\zeta_{23})$. Thus we have $\theta_{11} \in S$. Hence, from Theorem \ref{TA3}, there are integers $a_0, \ldots, a_{11}$ such that 
\begin{equation}\label{e1}
\theta_{11}=a_0N+a_1f_1+\ldots+a_{11}f_{11}.
\end{equation}
For any prime ideal $\mathfrak{p}$ of $\Q(\zeta_{23})$ of residue degree $11$ the decomposition group is 
$$D_{\mathfrak{p}}=\{ \sigma_2,\sigma_4,\sigma_8,\sigma_{16},\sigma_9,\sigma_{18},\sigma_{13},\sigma_3,\sigma_6,\sigma_{12},\sigma_1  \}.$$
Thus $\{\sigma_1,\sigma_5\}$ is a complete set of coset representatives of $G/D_{\mathfrak{p}}$. Note that any other set of coset representatives of $G/D_{\mathfrak{p}}$ is a multiple by an element of $D_{\mathfrak{p}}$. Hence, without loss of generality we take 
\begin{equation}\label{e2}
\theta_{11}=\sigma_1+\sigma_5.
\end{equation}
Comparing the coefficients of $\sigma_i$ in equations (\ref{e1}) and (\ref{e2}) we obtain a contradiction as explained below. \\
Comparing the coefficient of $\sigma_1$ leads to $a_0=1$ and comparing the coefficient of $\sigma_{12}$ leads to $a_{11}=-1$. On the other hand comparing the coefficients of $\sigma_{11}$ gives
\begin{equation}\label{e3}
1+a_1+a_2+\ldots+a_{10}=0
\end{equation}
and from the coefficients of $\sigma_{22}$ we obtain 
\begin{equation}\label{e4}
1+a_1+a_2+\ldots+a_{11}=0.
\end{equation}
Equations ($\ref{e3})$ and ($\ref{e4}$) together give $a_{11}=0$ which contradicts to $a_{11}=-1$. Thus $\theta_{11} \not \in S$ and it follows that the class group of $\Q(\zeta_{23})$ is not generated by the classes of prime ideals of residue degree $11$.\\
Next, we show that the class group of $\Q(\zeta_{23})$ is not generated by the classes of prime ideals of residue degree $2$. If $\mathfrak{p}$ is a prime ideal of residue degree $2$, then the decomposition group at $\mathfrak{p}$ is $$D_{\mathfrak{p}}=\{\sigma_1,\sigma_{22}\}.$$ 
As done earlier, we may assume that 
\begin{equation}\label{e5}
\theta_2=\sigma_1+\ldots +\sigma_{11}.
\end{equation}
From Theorem \ref{T3} $\theta_2$ annihilates the class group of $\mathbb{Q}(\zeta_{23})$ and thus $\theta_2 \in S$. From Theorem \ref{TA3}, there are integers $a_0, \ldots, a_{11}$ such that
\begin{equation}\label{e6}
\theta_2=a_0N+a_1f_1+\ldots+a_{11}f_{11}.
\end{equation}
The equations (\ref{e5}) and (\ref{e6}) leads to an inconsistent system of equation in $\sigma_i, 1 \leq i \leq 22$ and this is summarized below.\\
Coefficients of $\sigma_1$ gives $a_0=1$, coefficients of $\sigma_8$ gives $a_7=0$ and coefficients of $\sigma_{12}$ gives $a_{11}=-1$. Using these, the coefficients of $\sigma_4$ gives $a_3=1$ and coefficients of $\sigma_6$ gives $a_5=1$. Using these values in the relations obtained from coefficients of $\sigma_{11}$ and $\sigma_{22}$ further leads to $a_9=0$. Now coefficients of $\sigma_{14}$ leads to $a_4=-1$. In same way, coefficients of $\sigma_2$ leads to $a_1=-1$, coefficients of $\sigma_{10}$ leads to $a_6=-1$. Further the coefficients of $\sigma_{20}$ gives $a_{10}=0$. Now we see that the coefficients of $\sigma_3$ and that of $\sigma_{21}$ lead to the inconsistent system
$$a_2+a_8=-1 \mbox{ and } a_2+a_8=0.$$
Thus we have proved that
$$\mathscr{R}_{\mathbb{Q}}^{\mathbb{Q}(\zeta_{23})}=\{1\}.$$

\begin{rem}
Proceeding along the same line, we have made computations for $L=\Q(\zeta_{29})$ and $K=\Q$ and found that $\mathscr{R}_K^L=\{1\}$.
\end{rem}

Now we give a complete description of solvability of norm equations for $\mathbb{Q}(\zeta_{23})/ \mathbb{Q}$. Since the splitting type of any rational prime $q$ in $\mathbb{Q}(\zeta_{23})$ is well understood (see chapter 3 in \cite{LW}), the description in Theorem \ref{T4} is the best one can expect.
\begin{thm}\label{T4}
For any $a \in \mathbb{Q}$ the norm equation 
$$N_{\mathbb{Q}}^{\mathbb{Q}(\zeta_{23})}(x)=a$$ is solvable if and only if $\upsilon_p(a)$ is a multiple of the residue degree of $p$ for all primes $p$ and $a \geq 0$. Consequently the knot number of $\mathbb{Q}(\zeta_{23})$ is $1$.
\end{thm}
\begin{proof}
Since $\mathbb{Q}(\zeta_{23})$ is totally complex, the values of norm map are non-negative. It suffices to show that for each prime $q$ there is an $\alpha \in \mathbb{Q}(\zeta_{23})$ such that 
\begin{equation}\label{e7}
N_{\mathbb{Q}}^{\mathbb{Q}(\zeta_{23})}(\alpha)=p^f, 
\mbox{ where }f \mbox{ is the residue degree of } p.
\end{equation}

For $p=23$, equation (\ref{e7}) holds for $\alpha=1-\zeta_{23}$. When $p$ is a prime of residue degree $f>1$, then $p$ is unramified and existence of an $\alpha$ satisfying equation (\ref{e7}) is guaranteed  from the Theorem \ref{T2}.

Now assume that $p$ splits completely in $\mathbb{Q}(\zeta_{23})$, and let $\mathfrak{p}$ be a prime ideal of $\mathbb{Q}(\zeta_{23})$. If $\mathfrak{p}$ is principal then we are done. The class number of $\Q(\zeta_{23})$ is $3$ (see \cite{LW,JM}). Consequently there is a $\beta \in \Q(\zeta_{23})$ such that 
$$N_{\mathbb{Q}}^{\mathbb{Q}(\zeta_{23})}(\beta)=p^3.$$
On the other hand 
$$N_{\mathbb{Q}}^{\mathbb{Q}(\zeta_{23})}(p)=p^{22}.$$
Let $s,t \in \mathbb{Z}$ be such that $3s+22t=1$ then equation (\ref{e7}) holds for $\alpha =\beta^sp^t$.
\end{proof}

From the proof of Theorem \ref{T4} it also follows that if the class number $h_L$ of $L$ and the degree $[L:\mathbb{Q}]$ are coprime then the norm equation 
$$|N_{\mathbb{Q}}^{L}(\alpha)|=p^f, 
\mbox{ where }f \mbox{ is the residue degree of } p$$ is solvable for each prime $p$.

\section{Concluding remarks}
The purpose of this article is to convince the reader that the problem ``whether prime ideals of residue degree bigger than one are well distributed across the ideal class group or not" is a useful problem. In general, we are not aware of any method to tackle this problem; the analytic methods which successfully tackle the distribution of prime ideals of residue degree one are limited to prime ideals of residue degree one.

From the two examples we made computations for, it is tempting to look for some relation between  `$\mathscr{R}_{\mathbb{Q}}^{\mathbb{Q}(\zeta_{\ell})}=\{1\}$' and `$h_{\ell}^+=1$'.

The study carried out here is in the spirit that `look at subfields of $L$ to get information on $L$'. Such studies has been carried out earlier too, for example see \cite{MR}.\\

\par
\par
\noindent{\bf Acknowledgements.} The author would like to express his gratitude to Prof. Dipendra Prasad for some very fruitful discussions and making some corrections in earlier versions.

\end{document}